\documentclass[a4paper]{article}

\usepackage[english]{babel}
\usepackage[utf8]{inputenc}
\usepackage{amsmath}

\usepackage[a4paper,top=3cm,bottom=2cm,left=3cm,right=3cm,marginparwidth=1.75cm]{geometry}

\usepackage{amsthm}
\usepackage{amsfonts}
\usepackage{graphicx}
\usepackage[colorinlistoftodos]{todonotes}
\usepackage[colorlinks=true, allcolors=blue]{hyperref}
\usepackage{url}

\newtheorem{lem}{Lemma}
\newtheorem{proposition}{Proposition}
\newtheorem{theorem}{Theorem}
\newtheorem{rmk}{Remark}
\newtheorem{example}{Example}

\title{Primality tests, linear recurrent sequences and the Pell equation}

\author{Danilo Bazzanella$^*$, Antonio Di Scala$^*$, Simone Dutto$^*$, Nadir Murru$^{**}$ \\
$^*$ Politecnico of Turin, Corso Duca degli Abruzzi 24, Torino, 10129, ITALY \\
$^{**}$ University of Turin, Via Carlo Alberto 10, Torino, 10123, ITALY}
\date{}

\begin{document}

\maketitle

\begin{abstract}
We study new primality tests based on linear recurrent sequences of degree two exploiting a matricial approach. The classical Lucas test arises as a particular case and we see how it can be easily improved. Moreover, this approach shows clearly how the Lucas pseudoprimes are connected to the Pell equation and the Brahamagupta product. We also introduce a new specific primality test, which we will call generalized Pell test. We perform some numerical computations on the new primality tests and, for the generalized Pell test, we do not any pseudoprime up to $10^{10}$.
\end{abstract}


\section{Introduction}
The Pell equation is one of the most famous and studied Diophantine equation, it is
\[x^2 - D y^2 = 1\]
for $D$ non--square integer. Recently, its properties have been exploited in cryptographic applications for defining an RSA--like cryptosystem with multi--factor modulus \cite{NUMTA}. In this paper, we show how it is connected to primality tests, allowing to define new ones that appear to be very interesting.

One of the most classical primality test is based on the Little Fermat's Theorem, i.e., the Fermat test. It is known that there are infinitely many composite numbers that pass the test to every base \cite{AGP94}. However, it is possible to define a stronger test considering that for $p = 2^r s + 1$ prime, then
\[ a^s \equiv 1 \pmod p \quad \text{or} \quad a^{2^k s} \equiv -1 \pmod p \]
for any $a \in \mathbb Z_p^*$ and some $0 \leq k < r$. An odd composite number satisfying this condition is a \emph{strong pseudoprime} to base $a$ and it is known that there are no strong pseudoprimes to all bases. See \cite{PSW80} for a classical study on these pseudoprimes. The Baillie--PSW primality test combines the above test with the Lucas test \cite{BW80}; an overlap between strong and Lucas pseudoprimes has not been found so far. However, it is conjectured that there are infinitely many Baillie--PSW pseudoprimes and Pomerance gave an idea for constructing them \cite{Pom84}. Some calculations about the search of Baillie--PSW pseudoprimes can be also found in \cite{CG03}. The Lucas test is based on some properties of the Lucas sequence. Given two integers $P$ and $Q$ the Lucas sequence is defined by
\[ \begin{cases}  U_0 = 0, U_1 = 1 \cr U_k = P U_{k-1} + Q U_{k-2} \end{cases} \]
for any $k \geq 2$ and it can be evaluated by means of
\[ U_k = \cfrac{\alpha^k - \beta^k}{\alpha - \beta}, \]
where $\alpha, \beta$ are the roots of the characteristic polynomial. The Lucas test is based on the fact that when $p$ is prime, we have
\[ U_{p-1} \equiv 0 \pmod p \quad \text{or} \quad U_{p+1} \equiv 0 \pmod p \]
when $\left( \cfrac{D}{p} \right) = 1$ or $\left( \cfrac{D}{p} \right) = -1$ (Jacobi symbols), respectively, for $D = P^2 - 4 Q$. Thus, the Lucas pseudoprimes, with parameters $P$ and $Q$, are the odd composite integers $n$ such that
\begin{equation} \label{eq:Lucas}
U_{n-\left( D/n \right)} \equiv 0 \pmod n. 
\end{equation}
The Lucas pseudoprimes have been widely studied, see, e.g., \cite{DF88, GP91, Pom10, Som09, Suw12}. Some authors also studied primilaty tests using more general linear recurrence sequences \cite{Gra10, LS07}.

In \cite{DMS19}, the authors highlighted how the Lucas test can be introduced in an equivalent way by means of the Brahmagupta product and the Pell equation. We recall here some facts.

It is well--known that given two solutions $(x_1, y_1)$ and $(x_2, y_2)$ of the Pell equation, then the Brahmagupta product 
\[ (x_1, y_1) \otimes (x_2, y_2) = (x_1 x_2 + D y_1 y_2, x_1 y_2 + x_2 y_1) \]
yields to another solution of the Pell equation. For a complete survey, we refer the reader to \cite{Bar03}. 
Given a ring $\mathcal R$, we can consider the Pell conic
\[ \mathcal C = \{ (x,y) \in \mathcal R \times \mathcal R : x^2 - D y^2 = 1 \} \]
and $(\mathcal C, \otimes)$ is a group with identity $(1,0)$. Moreover, when $\mathcal R = \mathbb Z_p$, the order of $\mathcal C$ depends on $D$ to be or not a quadratic residue. In particular, we have $\lvert \mathcal C \rvert = p - 1$ if $D$ is a quadratic residue in $\mathbb Z_p$ and $\lvert \mathcal C \rvert = p + 1$ if not, see, e.g., \cite{MV92}. This property allows to construct a primality test. In \cite{DMS19}, the authors defined the \emph{Pell pseudoprimes}, as the odd composite integers $n$ such that
\begin{equation*}
y_n \equiv 0 \pmod n
\end{equation*}
with $(x_n, y_n) = (\tilde x, \tilde y)^{\otimes n - (D/n)}$, where $D$ and $(\widetilde{x}, \widetilde{y}) \in \mathcal C$ are the parameters of the test, for $\mathcal R = \mathbb Z_n$. With this definition, we have an equivalence between the Lucas and Pell test described in the following theorem.

\begin{theorem} 
On the one hand, if $n$ is a Lucas pseudoprime with parameters $P > 0$ and $Q = 1$, then $n$ is a Pell pseudoprime with parameter $\widetilde x \equiv P/2 \pmod n$, $\widetilde y \equiv 1/2 \bmod n$, and $D = P^2 - 4$. 
On the other hand, if $n$ is a Pell pseudoprime with parameter $\widetilde x$, $\widetilde y$, and $D$, then $n$ is a Lucas pseudoprime with parameters $P = 2 \widetilde x$, and $Q = 1$ \cite{DMS19}.
\end{theorem}

Considering the order of the Pell conic over finite fields $\mathbb Z_p$, we can clearly define a stronger test. Hence, we define the \emph{strong Pell pseudoprimes} as the odd composite integers $n$ such that
\[ (\tilde x, \tilde y)^{\otimes n - (D/n)} \equiv (1, 0) \pmod n, \]
where $D$ and $(\widetilde{x}, \widetilde{y}) \in \mathcal C$ are the parameters of the test, for $\mathcal R = \mathbb Z_n$.

\begin{rmk}
In \cite{Lem03}, the author highlighted the properties of the Pell conic for constructing a primality test, but only focused on a Pell conic of the kind $x^2 - Dy^2 = 4$, as well as in \cite{Ham12}, where the author focused on $x^2 + 3 y^2 = 4$ for testing numbers of the form $3^n h \pm 1$.
Moreover, we would like to point out that sometimes the term Pell pseudoprimes is used for the Lucas pseudoprimes with parameters $P = 2$ and $Q = -1$, since for these parameters the sequence $U_n$ is known as the Pell sequence (A000129 in OEIS \cite{Slo}). We have also found a different definition of Pell pseudoprimes that are the odd composite integers $n$ such that
\begin{equation*}
 U_n \equiv \left( \cfrac{2}{n} \right) \pmod n
\end{equation*}
for $P = 2$ and $Q = -1$ (A099011 in OEIS).
\end{rmk}

In this paper, firstly, in section \ref{sec:mat}, we show how many primality tests based on linear recurrent sequences of order 2 can be introduced from a matricial point of view. The Lucas test and the Pell test, as well as their connection, arise as particular cases. Moreover, in this way, we are able to introduce a generalized Pell test based on the quotient ring $\mathcal R[t] / (t^2 -D)$, for $D \in \mathcal R$, which also has an analogue via Lucas sequence as we will see. Then, in section \ref{sec:exp}, we perform numerical experiments and comparison between some primality tests with a special focus on the generalized Pell test. In particular, we will show a method for the choice of the parameters, inspired to the Selfridge method, that produces very promising results. Indeed, we did not found any composite numbers that pass the generalized Pell test, with such method for the choiche of parameters, up to $10^{10}$.

\section{Pseudoprimes with matrices} \label{sec:mat}

Given a matrix $M \in \mathbb Z^{2 \times 2}$, we can consider the linear recurrence sequences $(\widetilde U_k)_{k \geq 0}$ and $(\widetilde V_k)_{k \geq 0}$ defined by
\begin{equation*}
  \begin{pmatrix} \widetilde V_k \cr \widetilde U_k \end{pmatrix} := M^k \begin{pmatrix} 1 \cr 0 \end{pmatrix}.
\end{equation*}
The following lemma provides a primality test based on these sequences, the Lucas and strong Pell tests arise for particular choices of $M$.
Hence, this lemma will allow to highlight many primality tests based on linear recurrence sequences of order 2 and the connection between the Lucas test and the Pell conic.

\begin{lem} \label{lem:main}
Let $\Delta$ be the discriminant of the characteristic polynomial of $M \in \mathbb Z^{2 \times 2}$, if $p$ is prime and $\det M \not= 0 \pmod p$, then
\begin{enumerate}
\item $\widetilde U_{p-1} \equiv 0 \pmod p$ and $\widetilde V_{p-1} \equiv 1 \pmod p$, when $\sqrt{\Delta} \in \mathbb Z_p^*$;
\item $\widetilde U_{p+1} \equiv 0 \pmod p$ and $\widetilde V_{p+1} \equiv \det M \pmod p$, when $\sqrt{\Delta} \not\in \mathbb Z_p^*$.
\end{enumerate}
\end{lem}
\begin{proof}
Let $\alpha, \beta$ be the roots of the characteristic polynomial of $M$, we have that $M$ is similar to the diagonal matrix
\begin{equation*}
\begin{pmatrix} \alpha & 0 \cr 0 & \beta \end{pmatrix}.
\end{equation*}
Thus, when $\sqrt{\Delta} \in \mathbb Z_p^*$, we also have $\alpha, \beta \in \mathbb Z_p^*$ and $M^{p-1}$ is the identity matrix modulo $p$ by the Little Fermat's Theorem, then
\begin{equation*}
 \begin{pmatrix} \tilde V_{p-1} \cr \tilde U_{p-1} \end{pmatrix} = M^{p-1} \begin{pmatrix} 1 \cr 0 \end{pmatrix} \equiv \begin{pmatrix} 1 \cr 0 \end{pmatrix} \pmod p.
\end{equation*}
When $\sqrt{\Delta} \not\in \mathbb Z_p^*$, by the Frobenius morphism we have $\alpha^p = \beta$, $\beta^p = \alpha$ and
\begin{equation*}
 \begin{pmatrix} \tilde V_{p+1} \cr \tilde U_{p+1} \end{pmatrix} = M^p \cdot M \begin{pmatrix} 1 \cr 0 \end{pmatrix} \equiv \det M \begin{pmatrix} 1 \cr 0 \end{pmatrix} \pmod p
\end{equation*}
\end{proof}

In the following we see that the Lucas and Pell tests arise as particular cases of the previous Lemma.

\begin{lem} \label{prop:LC}
Given 
\begin{equation*}
L = \begin{pmatrix} P & -Q \cr 1 & 0 \end{pmatrix}, \quad C = \begin{pmatrix} \widetilde{x} & D \widetilde{y} \cr \widetilde{y} & \widetilde{x} \end{pmatrix},
\end{equation*}
we have
\begin{equation*}
 L^k \begin{pmatrix} 1 \cr 0 \end{pmatrix} = \begin{pmatrix} U_{k+1} \cr U_k \end{pmatrix}, \quad  C^k \begin{pmatrix} 1 \cr 0 \end{pmatrix} = \begin{pmatrix} x_k \cr y_k \end{pmatrix}
\end{equation*}
where $(U_k)_{k \geq 0}$ is the Lucas sequence with characteristic polynomial $t^2 - P t + Q$, $(\widetilde{x}, \widetilde{y}) \in \mathcal C$ (for any ring $\mathcal R$), and $(x_k, y_k) = (\widetilde{x}, \widetilde{y})^{\otimes k}$, for any $k \geq 0$.
\end{lem}
\begin{proof}
Let us denote with $L^k_{ij}$ the entry $(i,j)$ of the matrix $L^k$.
It is well--known that the entries of $L^k$ are linear recurrence sequences that recur with the characteristic polynomial of $L$, i.e., $t^2 - P t + Q$.
Observing that $L_{11}^0 = 1$, $L_{11}^1 = P$ and $L_{21}^0 = 0$, $L_{21}^1 = 1$ , we have 
\begin{equation*}
 L^k \begin{pmatrix} 1 \cr 0 \end{pmatrix} = \begin{pmatrix} U_{k+1} \cr U_k \end{pmatrix}
\end{equation*}
for any $k \geq 0$.

Given $(\widetilde{x}, \widetilde{y}) \in \mathcal C$, the sequences $(x_k)_{k \geq 0}$ and $(y_k)_{k \geq 0}$ defined by $(x_k, y_k) = (\widetilde{x}, \widetilde{y})^{\otimes k}$ can be also evaluated by
\begin{equation*}
 (\widetilde{x} + \sqrt{D} \widetilde{y})^k = x_k + \sqrt{D} y_k 
\end{equation*}
from which it is straightforward to obtain
\begin{equation*}
\begin{cases} x_{k+1} = \widetilde{x} x_k + D \widetilde{y} y_k \cr y_{k+1} = \widetilde{y} x_k + \widetilde{x} y_k \end{cases},
\end{equation*}
i.e.,
\begin{equation*}
 \begin{pmatrix} \widetilde{x} & D \widetilde{y} \cr \widetilde{y} & \widetilde{x} \end{pmatrix} \begin{pmatrix} x_k \cr y_k \end{pmatrix} = \begin{pmatrix} x_{k+1} \cr y_{k+1} \end{pmatrix}.
\end{equation*}
Observing that $(x_0, y_0)$ is the identity of $(\mathcal C, \otimes)$, i.e., $(1,0)$, we have
\begin{equation*}
 C^k \begin{pmatrix} 1 \cr 0 \end{pmatrix} = \begin{pmatrix} x_k \cr y_k \end{pmatrix}
\end{equation*}
for any $k \geq 0$.
\end{proof}

Now, we can see that the strong Pell test is connected with a stronger version of the Lucas test, which we will call \emph{double Lucas test}.
Indeed, by Lemma \ref{lem:main} and Lemma \ref{prop:LC}, if $p$ is prime, then
\begin{equation*}
 U_{p-1} \equiv 0 \pmod p, \quad U_p \equiv 1 \pmod p 
\end{equation*} 
or
\begin{equation*}
 U_{p+1} \equiv 0 \pmod p, \quad U_{p+2} \equiv Q \pmod p 
\end{equation*}
for $\sqrt{P^2 - 4 Q} \in \mathbb Z_p^*$ or $\sqrt{P^2 - 4 Q} \not\in \mathbb Z_p^*$, respectively. We call double Lucas pseudoprimes the odd composite numbers that satisfy the above conditions.

Since $\det C = 1$ and $\det L = Q$, the matrices $C$ and $L$ are similar only if $Q = 1$.
In this case, we can consider the matrix
\begin{equation*}
 R_1 = \begin{pmatrix} 1 & P \cr 0 & 2 \end{pmatrix}, 
\end{equation*}
 and we have 
$$R_1^{-1} L R_1 = \begin{pmatrix} P/2 & P^2/2 -2 \cr 1/2 & P/2 \end{pmatrix},$$
choosing $\widetilde{x} = P/2$, $\widetilde{y} = 1/2$, $D = P^2 - 4$, we get $R_1^{-1} L R_1 = C$.
In other words, if $n$ is a double Lucas pseudoprime, for parameters $P$ and $Q=1$, then $n$ is a strong Pell pseudoprime, for parameters $\widetilde{x} = P/2$, $\widetilde{y} = 1/2$, $D = P^2 - 4$.
Let us note that $2$ must be invertible in $\mathbb Z_n$.

On the other hand, given
\begin{equation*}
 R_2 = \begin{pmatrix} 1 & -\widetilde{x} \cr 0 & \widetilde{y} \end{pmatrix} 
\end{equation*}
we have
\begin{equation*}
 R_2^{-1} C R_2 = \begin{pmatrix} 2 \widetilde{x} & -\widetilde{x}^2 + D \widetilde{y}^2 \cr 1 & 0 \end{pmatrix}, 
\end{equation*}
which is $L$ for $P = 2 x_1$ and $Q = 1$.
This means that if $n$ is a strong Pell pseudoprime, for parameters $\widetilde{x}, \widetilde{y}, D$, then $n$ is a double Lucas pseudoprime, for parameters $P = 2 x_1$ and $Q = 1$.
Note that in this case $D$ is not necessarily equal to $P^2 - 4$ (the discriminant of the characteristic polynomial of the Lucas sequence), this happens only for $\widetilde{y} = \pm 1/2$.
However, since $D = (\widetilde{x}^2-1)/\widetilde{y}^2$ and $P^2 - 4 = 4 (\widetilde{x}^2 - 1)$, $D$ is a quadratic residue in $\mathbb Z_n$ if and only if $P^2 - 4$ is. We summarize this in the following proposition.

\begin{proposition}
If $n$ is a double Lucas pseudoprime for the parameters $P$ and $Q = 1$, then $n$ is a strong Pell pseudoprime for the parameters $\widetilde{x} = P/2$, $\widetilde{y} = 1/2$, $D = P^2 - 4$. \\
If $n$ is a strong Pell pseudoprime for the parameters $\widetilde{x}$, $\widetilde{y}$ and $D$, then $n$ is a double Lucas pseudoprime for the parameters $P = 2 \widetilde{x}$ and $Q = 1$.
\end{proposition}

Let us note that, fixed the parameters $P$ and $Q = 1$ for the Lucas test (for checking, e.g., the primality of all the integers in a certain range), there is not a corresponding strong Pell test with fixed parameters $D$, $\widetilde{x}$ and $\widetilde{y}$ as integer numbers.
Indeed, given any $P$ and $Q = 1$, we have seen that $\widetilde{x} = P/2$, $\widetilde{y} = 1/2$, $D = P^2 - 4$ are the corresponding parameters of the strong Pell test, but these values depend on the integer $n$ we are testing (remember that in this context $1/2$ is the inverse of $2$ in $\mathbb Z_n$).

Moreover, in general, we are not able to fix the integer parameters $D, \widetilde{x}, \widetilde{y}$ in the strong Pell test for checking the primality of all the integers in a given range, because it is necessary that $\widetilde{x}^2 - D \widetilde{y}^2 \equiv 1 \pmod n$ and this can not be true for any integer $n$.
For overcoming these issues, the use of a parametrization of the conic $\mathcal C$ can be helpful.
In \cite{BM16}, the authors provided the following map
\begin{equation*}
 \Phi : \begin{cases} \mathcal R \cup \{\alpha\} \rightarrow \mathcal C \cr a \mapsto \left( \cfrac{a^2 + D}{a^2 - D}, \cfrac{2 a}{a^2 - D} \right), \quad a \not= \alpha \cr \alpha \mapsto (1,0) \end{cases} 
\end{equation*}
where $\alpha \not \in \mathcal R$ is the point at the infinity of such a parametrization of $\mathcal C$.
When $\mathcal R$ is a field and $t^2 - D$ is irreducible in $\mathcal R$, the map is always defined, otherwise there are values of $a$ such that $\Phi(a)$ can not be evaluated.
In this way, we can consider the strong Pell test with fixed parameters $D$ and $a$, in the sense that $\widetilde{x} = (a^2 + D)/(a^2 - D)$ and $\widetilde{y} = 2 a/(a^2 - D)$.

\begin{example} \label{exm:P-even}
Given $P = 4$ and $Q = 1$, the Lucas pseudoprimes up to 5000 are
\begin{equation*}
 65, 209, 629, 679, 901, 989, 1241, 1769, 1961, 1991, 2509, 2701, 2911, 3007, 3439, 3869,
\end{equation*}
whereas the double Lucas pseudoprimes are
\begin{equation*}
 209, 901, 989, 2701, 2911, 3007, 3439.
\end{equation*}
When $P$ is even, we are always able to find an equivalent strong Pell test, providing all the same pseudoprimes of the double Lucas test.
Indeed, it is sufficient to choice $D$ and $a$ such that $(a^2 + D) / (a^2 - D)$ is the integer number $P/2$.
For instance in this case, taking $D = 3$ and $a = 3$, we have $\widetilde{x} = 2$ and $\widetilde{y} = 1$.
\end{example}

\begin{rmk}
A double Lucas test with parameters $P$ and $Q = 1$ is equivalent to the strong Pell test with parameters $D = P^2 - 4$ and $a = P + 2$.
Indeed, in this case, exploiting the parametrization $\Phi$, we get $\widetilde{x} = P/2$ and $\widetilde{y} = 1/2$.
Note that using this method, the strong Pell test equivalent to the double Lucas test considered in Example \ref{exm:P-even} has parameters $D = 12$ and $a = 6$.
This means that there are strong Pell tests with different parameters which are equivalent to each others.
\end{rmk}


We conclude observing that the double Lucas test for any value of $P$ and $Q$ can be described in terms of the Barahmagupta product.
The Pell equation can be introduced over a general ring $\mathcal R$ considering the quotient ring $\mathcal A = \mathcal R[t] / (t^2 - D)$, for $D \in \mathcal R$.
The product of two elements $x_1 + y_1 t, x_2 + y_2 t \in \mathcal A$, i.e., $(x_1, y_1), (x_2, y_2) \in \mathcal A$, coincide with the Brahmagupta product and the elements of norm 1 define $\mathcal C$.
If we take $(x_1, y_1) \in \mathcal A$ with norm $Q$, considering $(x_n, y_n) := (x_1, y_1)^{\otimes n}$, we still have 
\begin{equation*}
C^n \begin{pmatrix} 1 \cr 0 \end{pmatrix} = \begin{pmatrix} x_n \cr y_n \end{pmatrix}
\end{equation*}
and we can still use the matrices $R_1$ and $R_2$ for passing from $L$ to $C$ and viceversa, without any restriction on the choice of $Q$.
Hence the double Lucas test is connected with the Brahmagupta product also for $Q \not = 1$. Hence, we define the \emph{generalized Pell pseudoprimes}, for the parameters $D, \tilde x, \tilde y$, as the odd composite integers $n$ such that
\[ \begin{cases} (\tilde x, \tilde y)^{\otimes n+1} \equiv (Q, 0) \pmod n,\quad \text{if $\cfrac{D}{n}$ = -1} \cr 
(\tilde x, \tilde y)^{\otimes n-1} \equiv (1, 0) \pmod n,\quad \text{if $\cfrac{D}{n}$ = 1} \end{cases}.\]

\section{Numerical experiments} \label{sec:exp}

In this section, we show the behaviour of some primality tests in terms of number of pseudoprimes that pass them. In particular, we first focus on the classical Lucas test and we show how the use of the double Lucas test decreases a lot of the number of composite integers that are stated primes. Then, we see how the use of matrices introduced in the previous section allow the definition of many new primality tests and we study them for some different values of the parameters. Similarly, we also study the strong Pell test. In these experiments, we will see that the performances of the above tests are very sensitive with respect to the values of the parameters. For this reason in subsection \ref{sub:self} we study these tests setting the parameters by using methods à la Selfridge. In facts, the Selfridge method was introduced for finding good values of the parameters of the Lucas and strong Lucas tests, see \cite{BW80} and observe that in OEIS the sequences of Lucas pseudoprimes (A217120) and strong Lucas pseduoprimes (A217255) are defined by using the parameters $P$ and $Q$ with the Selfridge method.  

\subsection{Tests with fixed parameters}

The Lucas test depends on two parameters $P$ and $Q$ that determine the Lucas sequence $(U_k)_{k \leq 0}$ used in equation \eqref{eq:Lucas} for testing the primality of an integer number. In Figure \ref{fig:Lucas}, we show the number of Lucas pseudoprimes up to $10^5$ for $-3 \leq P \leq 3$ and $-3 \leq Q \leq 3$, avoiding trivial choices ot the parameters like, e.g., $P=1$, $Q=1$ or $P=2$, $Q=-1$. For instance, we can see that for $P=-3$, $Q=-3$ there are 45 Lucas pseudoprimes, for $P= -3$, $Q= -2$ there are 94 Lucas pseudoprimes, and so on.

Similarly, the double Lucas test depends on the same two parameters $P$ and $Q$, since the double Lucas pseudoprimes are the odd composite integers $n$ satisfying
\[ \begin{cases} U_{n-1} \equiv 0 \pmod n \quad \text{and} \quad U_{n} \equiv 1 \pmod n, \quad \text{if $\cfrac{D}{n}$ = 1} \cr 
U_{n+1} \equiv 0 \pmod n \quad \text{and} \quad U_{n+2} \equiv Q \pmod n, \quad \text{if $\cfrac{D}{n}$ = -1}
\end{cases}. \]
In Figure \ref{fig:double}, we can see how the double Lucas test decreases the number of pseudoprimes with respect to the Lucas test. For instance, for $P = -3$ and $Q = -3$ there are only 2 double Lucas pseudoprimes up to $10^5$, against the 45 Lucas pseudoprimes; for $P = -3$ and $Q = -2$ there are no double Lucas pseudoprimes (the first one is 220729 and the second one is 334153, no further pseudoprimes are found up to $5 \times 10^5$), whereas we found 94 Lucas pseudoprimes. However, we also have to observe that in some cases, the double Lucas test does not provide great improvements in this sense, for example for $P = -3$ and $Q = 1$ we found $50$ double Lucas pseudoprimes and $91$ Lucas pseudoprimes; for $P = -3$ and $Q = 2$ we found the same number of pseudoprimes. 

\begin{figure} \label{fig:Lucas}
\includegraphics[scale=0.52]{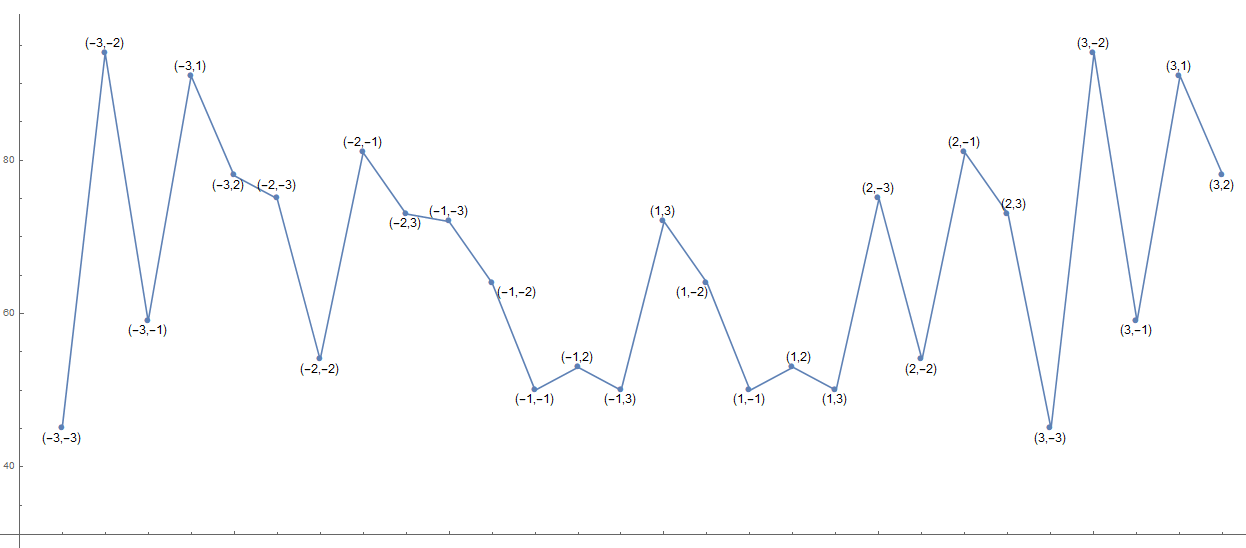}
\caption{Number of Lucas pseudoprimes up to $10^5$ for different values of $(P, Q)$.}
\end{figure}

\begin{figure} \label{fig:double}
\includegraphics[scale=0.52]{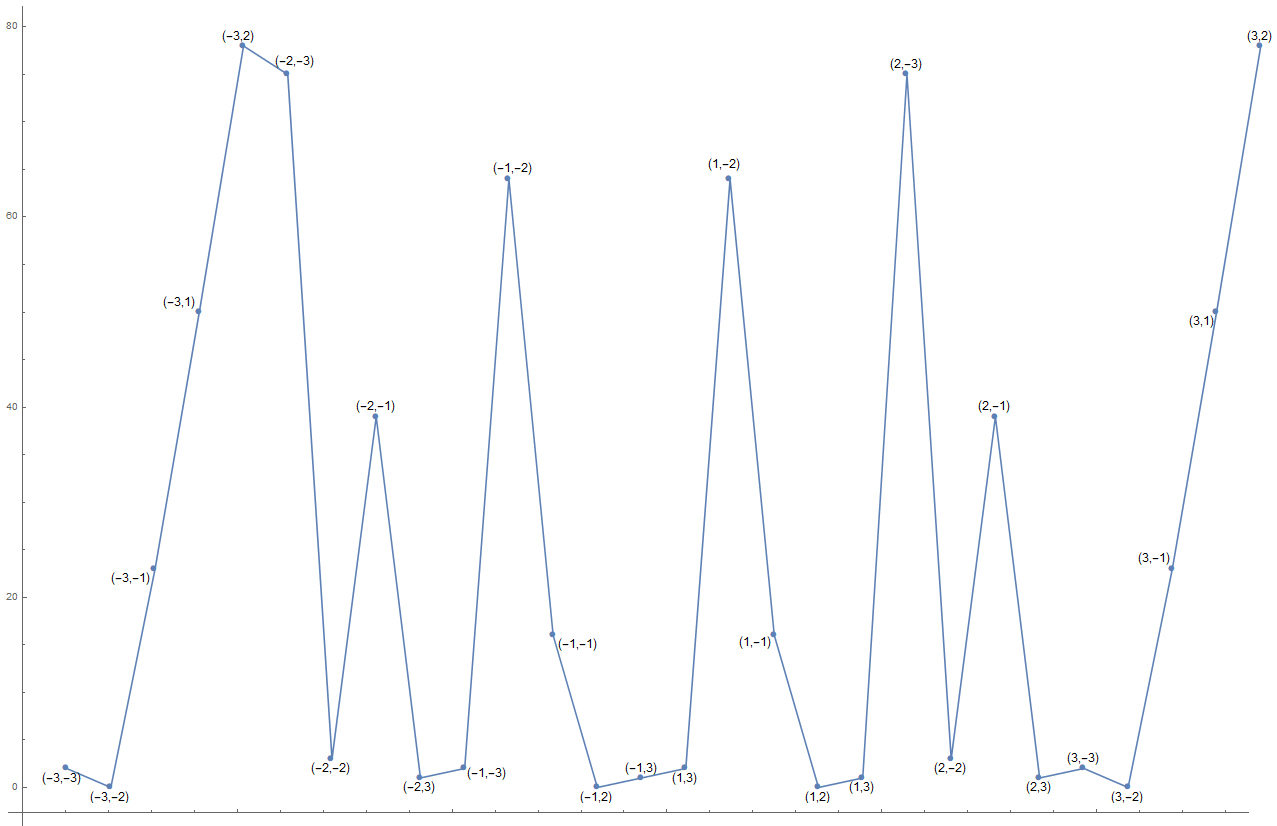}
\caption{Number of double Lucas pseudoprimes up to $10^5$ for different values of $(P, Q)$.}
\end{figure}

Thanks to Lemma \ref{lem:main}, we are able to define primality tests based on linear recurrent sequences of degree two, different from the Lucas sequences. Specifically, if we take the matrix
\[ \begin{pmatrix} P & - Q \cr R & 0 \end{pmatrix} \]
and consider the sequences $(\tilde U_k)_{k \geq 0}$ and $(\tilde V_k)_{k \geq 0}$ defined by $M^k \begin{pmatrix} 1 \cr 0 \end{pmatrix} = \begin{pmatrix} \tilde V_k \cr \tilde U_k \end{pmatrix}$, we deal with linear recurrent sequences with characteristic polynomial $t^2 - P t + QR$ and initial conditions $\tilde U_0 = 0, \tilde U_1 = R$ and $\tilde V_0 = 1, \tilde V_1 = P$. Note that $(\tilde V_k)_{k \geq 0}$ is the same sequence used in the double Lucas test, whereas $(\tilde U_k)_{k \geq 0}$ is different due to the initial conditions.
In this case the pseudoprimes are the odd composite integers satisfying

\begin{equation} \label{eq:matrix-pseudo}
\begin{cases} U_{n-1} \equiv 0 \pmod n \quad \text{and} \quad U_{n} \equiv 1 \pmod n, \quad \text{if $\cfrac{D}{n}$ = 1} \cr 
U_{n+1} \equiv 0 \pmod n \quad \text{and} \quad U_{n+2} \equiv QR \pmod n, \quad \text{if $\cfrac{D}{n}$ = -1}
\end{cases}. 
\end{equation}

The corresponding primality test depends on the parameters $P, Q, R$.
In Figure \ref{fig:matrix}, we reported the number of pseudoprimes up to $10^5$ for different values of the parameters. In particular, we used $R = -3, -2, -1, 2, 3$ (note that $R=1$ corresponds to the double Lucas test) and random values for $P$ and $Q$ between $-9$ and $9$. We can observe that the results strongly depend on the values of the parameters. For instance for $R = -1$, $P = 1$, $Q = 2$ we have no pseudoprimes up to $10^5$ (the only psuedoprime up to $5 \times 10^5$ is 226801), on the contrary for $R = 2$, $P = 3$, $Q = 2$ we found 123 pseudoprimes.

\begin{figure} \label{fig:matrix}
\includegraphics[scale=0.52]{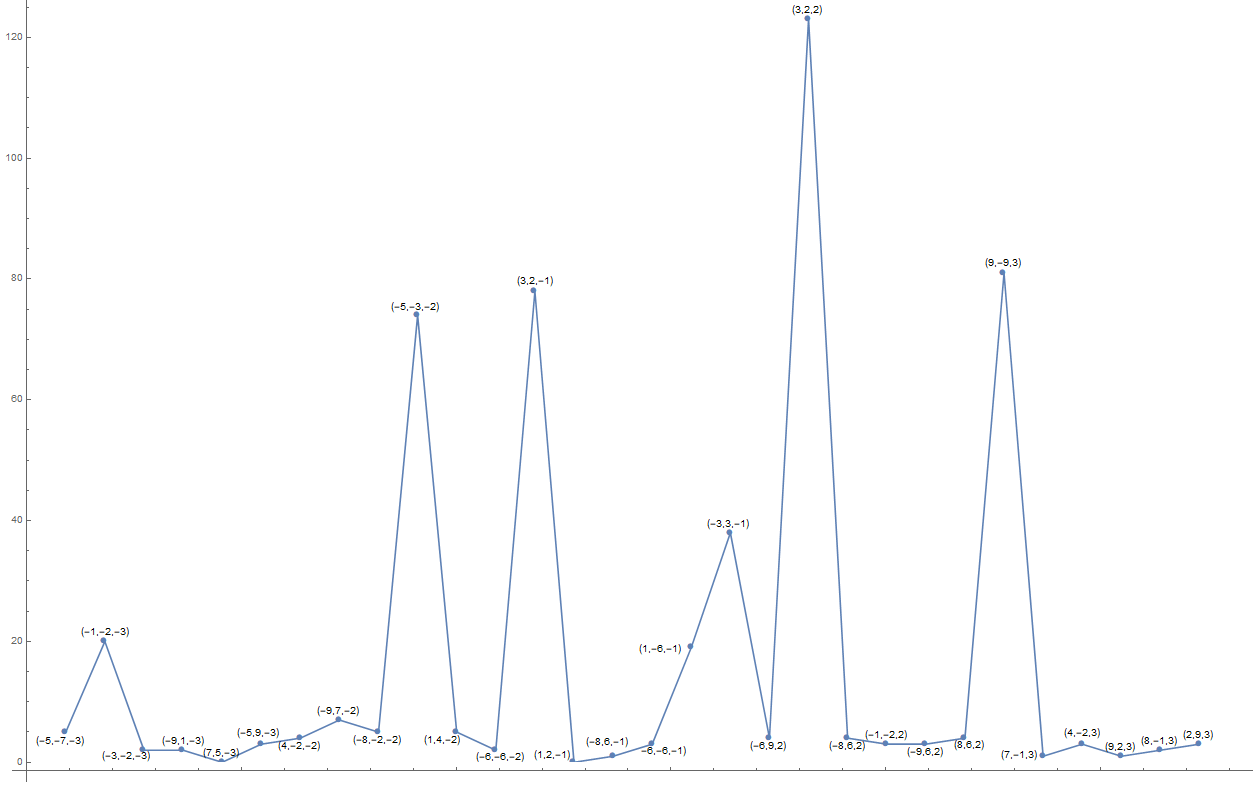}
\caption{Number of pseudoprimes up to $10^5$ defined by \eqref{eq:matrix-pseudo} for different values of $(P, Q, R)$.}
\end{figure}

Finally, we discuss the strong and generalized Pell pseudoprimes.  As we have shown in the previous section, they are connected to the double Lucas pseudoprimes. The strong Pell test with parameters $D$ and $a$ is equivalent to the double Lucas test with parameters $P = 2 (a^2 + D) / (a^2 -D)$ and $Q = 1$. The generalized Pell test depends on the parameters $D$, $\tilde x$, $\tilde y$ and there is not an equivalent double Lucas test with fixed parameters $P$ and $Q$. Thus, we only focus on some numerical experiments regarding this test. In Figure \ref{fig:gen-pell}, we show the number of generalized Pell pseudoprimes up to $10^5$ for $-3 \leq D \leq 3$ and $\tilde x$, $\tilde y$ randomly chosen between $-9$ and $9$. Also in this case, the performances are heavily affected by the choice of the parameters

\begin{figure} \label{fig:gen-pell}
\includegraphics[scale=0.52]{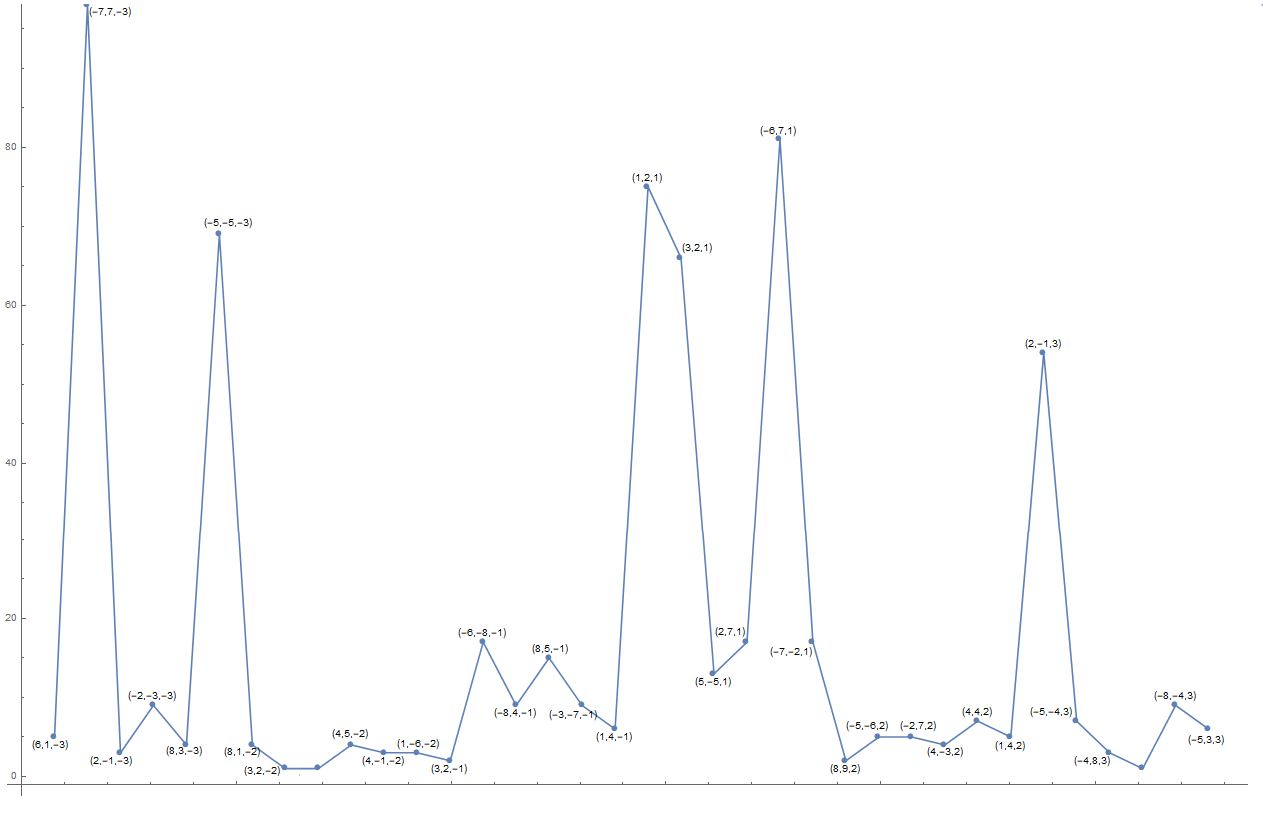}
\caption{Number of generalized Pell pseudoprimes up to $10^5$ for different values of $(\tilde x, \tilde y, D)$.}
\end{figure}

\subsection{Tests with Selfridge method} \label{sub:self}

In the previous section, we have seen that the performances of the studied primality tests are heavily affected by the choice of the parameters. The Selfridge method is a standard way for choosing the parameters of the Lucas test. Given the integer $n$ to test, the parameters of the Lucas test are chosen by the Selfridge method in the following way
\begin{itemize}

\item Set $P = 1$.

\item Set $D$ as the first integer in the sequence $5, -7, 9, -11, \ldots$ such that $\left( \cfrac{D}{n} \right) = -1$.

\item Set $Q = \cfrac{1 - D}{4}$. 

\end{itemize}

The sequence of Lucas pseudoprimes, using the Selfridge method, is 
\[323, 377, 1159, 1829, 3827, 5459, 5777, 9071, 9179, 10877, 11419, 11663, 13919, 14839, \ldots\]
see A217120 in OEIS. If we apply the Selfridge method for the choice of the parameters to the double Lucas test, we find the following sequence of pseudoprimes
\[ 5777, 10877, 75077, 100127, 113573, 161027, 162133, 231703, \ldots \]
which are the Frobenius pseudoprimes (A212423). Since the Selfridge method is a very standard and useful techniques for choosing the parameters in these primality tests, here we adapt the method to the new primality tests introduced previously. We will observe that the use of the Selfridge method with the primality test defined by \eqref{eq:matrix-pseudo} and with the generalized Pell pseudoprimes gives very interesting and powerful results. In the following we describe the adaptation of the Selfridge method to these tests.

Given the integer $n$ to test, for the primality test defined by \eqref{eq:matrix-pseudo}, choose the parameters in the following way:

\begin{itemize}

\item Set $P = 1$ and $R = 2$.

\item Set $D$ as the first integer in the sequence $-7, 9, -15, 17, -23, 25, -31, 33 \ldots$ such that $\left( \cfrac{D}{n} \right) = -1$.

\item Set $Q = \cfrac{1 - D}{8}$.

\end{itemize}

Using this method for the choice of the parameters, we did not find any pseudoprime up to $10^7$. Note that if we set $R = \pm 1$, we find the Frobenius pseudoprimes, thus we used $R = 2$ and we modified the sequence where searching $D$ in order to obtain an integer value for $Q$.

Given the integer $n$ to test, for the generalized Pell test, choose the parameters in the following way:

\begin{itemize}

\item Set $\tilde x = 3$ and $\tilde y = 2$.

\item Set $D$ as the first integer in the sequence $5, -7, 9, -11, \ldots$ such that $\left( \cfrac{D}{n} \right) = -1$.

\end{itemize}

Using this method, we did not find any pseudoprime up to $10^{10}$. 

In conclusion, the primality tests introduced in this paper, joint to the Selfridge method, appear to be very promising in terms of finding good primality tests. Indeed, usually, in the Lucas test and similar ones, there are small pseudoprimes (the first Lucas pseudoprime is 323 and the first Frobenius pseudoprime is 5777), whereas, for the generalized Pell test, the first pseudoprime must be greater than $10^{10}$. In future works, it should be interesting to find the first generalized Pell pseudoprime, as well as investigating different choices for $\tilde x$ and $\tilde y$. Moreover, it could be very itneresting to find some theoretical results about the distribution of generalized Pell pseudoprimes.

\section*{Acknowledgment}
A.~J. Di Scala is member of GNSAGA of INdAM and of DISMA Dipartimento di Eccellenza MIUR 2018-2022.\\
We would like to thank Dr. Tiziana Armano for the great help and support in the numerical experiments.

\end{document}